\documentclass[10pt]{amsart}

\usepackage{graphicx,amsmath,mathtools,amssymb,float}
\usepackage{epsfig,color,fancyhdr,setspace}
\usepackage{dsfont}
\usepackage{epstopdf}
\usepackage{import}
\usepackage{lineno}
\usepackage{stackrel}
\usepackage[paperwidth=8.5in, paperheight=11in,total={6.5in,8.5in}]{geometry}

\definecolor{darkgreen}{RGB}{0,100,0}

\newtheorem{theorem}{Theorem}[section]
\newtheorem{lemma}[theorem]{Lemma}

\newtheorem{example}[theorem]{Example}
\newtheorem{proposition}[theorem]{Proposition}
\newtheorem{remark}[theorem]{Remark}
\newtheorem{corollary}[theorem]{Corollary}
\newtheorem{conjecture}[theorem]{Conjecture}

\newtheorem{question}[theorem]{Question}

\numberwithin{equation}{section}
\numberwithin{figure}{section}

\begin{document}
	
	
	\title{On multiplying curves in the Kauffman bracket skein algebra\\ of the thickened four-holed sphere}
	
	\author{Rhea Palak Bakshi}
	\address{Department of Mathematics, The George Washington University, Washington DC, USA.}
	\email{rhea\_palak@gwu.edu}
	
	\author{Sujoy Mukherjee}
	\address{Department of Mathematics, The George Washington University, Washington DC, USA.}
	\email{sujoymukherjee@gwu.edu}
	
	\author{J\'{o}zef H. Przytycki}
	\address{Department of Mathematics, The George Washington University, Washington DC, USA \& University of Gda\'{n}sk, Poland.}
	\email{przytyck@gwu.edu}
	
	\author{Marithania Silvero}
	\address{Barcelona Graduate School of Mathematics at Universitat de Barcelona, Barcelona, Spain }
	\email{marithania@us.es}
	
	\author{Xiao Wang}
	\address{Department of Mathematics, The George Washington University, Washington DC, USA.}
	\email{wangxiao@gwu.edu}
	
	\begin{abstract}
		
		Based on the presentation of the Kauffman bracket skein module of the torus given by the third author in previous work,  
		Charles D. Frohman and R\u{a}zvan Gelca established a complete description of the multiplicative operation leading to a famous product-to-sum formula. 
		In this paper, we study the multiplicative structure of the Kauffman bracket skein algebra of the thickened four-holed sphere. We present an algorithm to compute the product of any two elements of the algebra, and give an explicit formula for some families of curves. We surmise that the algorithm has quasi-polynomial growth with respect to the number of crossings of a pair of curves. Further, we conjecture the existence of a positive basis for the algebra. 
	\end{abstract}

	
	
	\maketitle
	
	\section{Introduction}

	Skein modules were introduced in 1987 by the third author (see \cite{Pr1,Pr4}) as an invariant of embeddings of codimension two, modulo ambient isotopy.  In the case of 3-manifolds they generalize quantum invariants of classical links. In particular, the Kauffman bracket skein module and algebra are generalizations of the Kauffman bracket polynomial of links in the 3-sphere.
	
	In 1997,  Frohman and Gelca gave a famous product-to-sum formula for the Kauffman bracket skein algebra (KBSA) of the thickened torus \cite{FG}.
	In this paper we give an algorithm to multiply any two curves in $F_{0,4} \times I$.\footnote{In \cite{Dehn2}, Max Dehn denotes $F_{0,4}$ by $L_4$ and calls it the four-holed sphere. In \cite{FM}, Benson Farb and Dan Margalit call $F_{0,4}$, a lantern. Another name for $F_{0,4}$ is T-shirt. } We give several concrete applications of the algorithm. 
	In particular, we give a simple formula for mutiplying unit fractions\footnote{Of historical interest is the fact that unit fractions were the base for ancient Egyptian fraction calculus \cite{Gol}.} by $\frac{1}{0}$. 
	
	The paper is organized as follows. In the second section we recall background material: the Kauffman bracket skein module of 3-manifolds,
	the skein algebra structure for a thickened surface, and the exact structure (the generators and relations) of the algebras for the torus and the sphere with four holes, $F_{0,4}$. In Section \ref{SecFamilies}, we give closed formulas for the product of some families of curves yielding relatively simple formulas. Finally, in the fourth section we recall the Farey diagram and its basic properties. We then describe an algorithm to compute a product-to-sum formula for every pair of fractions. In the last part we describe future directions and open problems.

	\section{Preliminaries}\label{Section_Prelim}
	
	\subsection{The Kauffman bracket skein module}\label{Subsection_KBSM}
	
	The Kauffman bracket skein module (KBSM) of 3-manifolds is the most extensively studied object in the theory of skein modules. Let $M$ be an oriented $3$-manifold, $k$ a commutative ring with unity and $A$ a fixed invertible element in $k$. Let $\mathcal{L}_{fr}$ be the set of unoriented framed links in $M$ modulo topological equivalence, including the empty link, and $k\mathcal{L}_{fr}$ the free $k$-module generated by $\mathcal{L}_{fr}$. Denote by $S_{2, \infty}$ the submodule of $k\mathcal{L}_{fr}$ generated by all the skein expressions of the form: $$  \parbox{0.4cm}{\includegraphics[scale = 0.75]{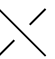}} - A \ \parbox{0.4cm}{\includegraphics[scale = 0.09]{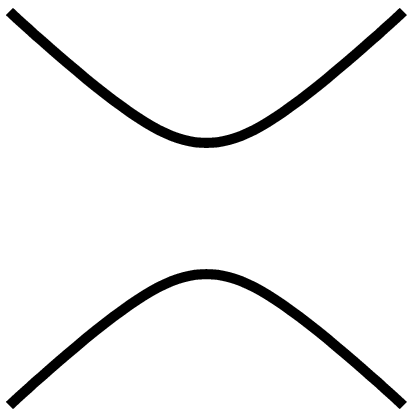}} - A^{-1} \ \parbox{0.4cm}{\includegraphics[scale = 0.09]{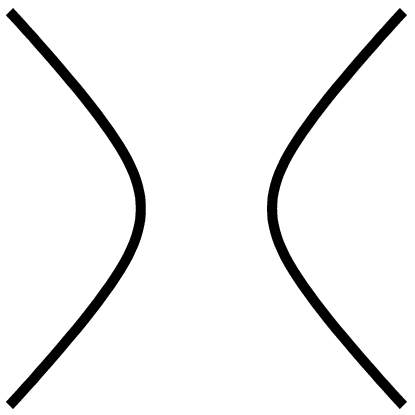}} \ and \ \bigcirc  \sqcup \ L + (A^2 + A^{-2}) L, $$ where $\bigcirc$ denotes the trivial framed knot and the skein triple in the first relation represents framed links which can be isotoped to identical embeddings except within the neighborhood of a single crossing, where they differ as shown. The change of the crossing $\parbox{0.4cm}{\includegraphics[scale = 0.75]{leftcross.eps}} $ to $\parbox{0.4cm}{\includegraphics[scale = 0.09]{L0.eps}}$ and $\parbox{0.4cm}{\includegraphics[scale = 0.09]{Linfty.eps}}$ is called A-smoothing and B-smoothing, respectively.
	
	The \textit{Kauffman bracket skein module} of $M$ is the quotient $\mathcal{S}_{2,\infty}(M; k, A) = k\mathcal{L}_{fr} / S_{2, \infty}$. For brevity, let $\mathcal{S}_{2,\infty}(M)$ denote $\mathcal{S}_{2,\infty}(M;k,A),$ when $k = \mathbb Z[A^{\pm 1}].$
	
 The KBSM for a few 3-manifolds have been computed in \cite{HP, Le, LT, Mar, Pr1}. If the $3$-manifold is the product of a surface and the interval, one can simply project the links onto the surface and work with their link diagrams.
	
	\begin{theorem}\cite{Pr1, Pr4}\label{TeoBasisFxI}
		Let $F$ be an oriented surface and $I = [0,1]$. Then, $\mathcal{S}_{2,\infty}(F \times I)$ is a free module generated by links (i.e. multicurves) including the empty link in $F$ with no trivial components. 
	\end{theorem}
	
	The generators in the statement of the above theorem are assumed to have blackboard framing. Let $F_{g,b}$ denote a surface of genus $g$ with $b$ boundary components. Theorem \ref{TeoBasisFxI} applies, in particular, to handlebodies, since $H_{b} = F_{0,b+1} \times [0,1]$, where $H_b$ is a handlebody of genus $b.$ 
	
	We enrich the KBSM of $F \times I$ with an algebra structure where the empty link is the identity element and the multiplication of two elements $L_1 \cdot L_2$ is defined by placing $L_1$ above $L_2$ that is, $L_{1} \subset F\times (\frac{1}{2}, 1)$ and $L_2 \subset F\times(0,\frac{1}{2})$. When we switch the order of the factors in the product of two curves the roles of $A$ and $A^{-1}$ are reversed.
	
The Kauffman bracket skein algebra of a surface times an interval was extensively discussed in \cite{BP, Bull1, FKL, PS1, PS2}. If $F$ is not a disc then the KBSA is an infinite dimensional module. As an algebra, however, it is finitely generated.
	
	\begin{theorem}\cite{Bull1}\label{TeoNumberGen}
		The algebra associated to $\mathcal{S}_{2,\infty}(F \times I)$ is finitely generated, and the minimal number of generators is $2^{rank(H_1(F))} - 1$. 
	\end{theorem}
	
	\begin{theorem}\cite{Pr3, PS1, PS2}\label{TeoCenter}
	The center of the algebra $\mathcal{S}_{2,\infty}(F \times I)$ is a subalgebra generated by the boundary curves of $F$.\footnote{In \cite{FKL} the center of the skein algebra is analyzed when $A$ is a root of unity.}
	\end{theorem}
	
	The exact structure of the KBSA has been computed for small surfaces $F_{g,b}$, namely, it is known for $(g,b) \in \{ (0,0), (1,0), (1,1), (1,2), (0,1), (0,2), (0,3), (0,4) \}$. In particular, the KBSA is commutative in the following cases:
	
	\begin{proposition}\cite{BP}  As an algebra,
		\begin{enumerate}
			
			\item $\mathcal{S}_{2, \infty}(F_{0,0} \times I;k,A) = \mathcal{S}_{2, \infty}(F_{0,1} \times I;k,A) = k$.
			\item $\mathcal{S}_{2, \infty}(F_{0,2} \times I;k,A) = k[x]$, where $x$ is a curve parallel to a boundary component.
			\item $\mathcal{S}_{2, \infty}(F_{0,3} \times I;k,A) = k[x,y,z]$, where $x,y,z$ are curves parallel to the boundary components.
		\end{enumerate}
	\end{proposition}
	
	In all other cases the algebra $S_{2,\infty}(F_{g,b} \times I)$ is noncommutative \cite{BP}.
	
We continue with the study of the KBSA($F_{0,4}\times I$), extending techniques developed by Frohman and Gelca in \cite{FG} for the torus. The remainder of this subsection is devoted to reviewing the description of the KBSA$(F_{1,0} \times I)$ and KBSA$(F_{0,4} \times I)$, as well as the relation between them based on the double branched cover $F_{1,0} \longrightarrow S^2$ branched along four points (see Figure \ref{abcdef}).

\begin{figure}
	\centering
	\includegraphics[scale = 0.4]{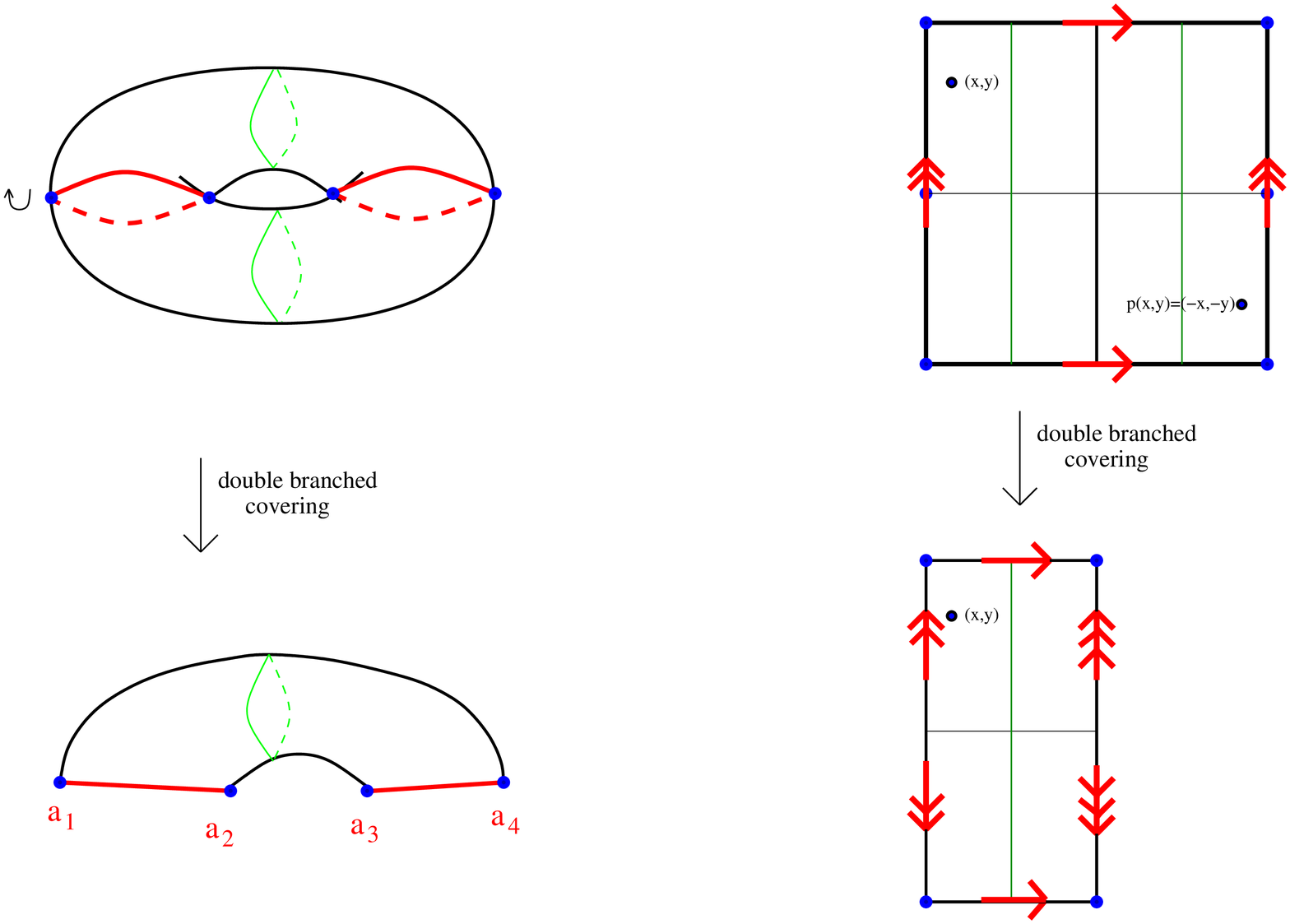}
	\caption{Double branched cover of $F_{1,0} \longrightarrow S^2 $ branched along four points. }
	\label{abcdef}
\end{figure}
	
The multicurves on the torus are parametrized by the pairs $(d,n)\in \mathbb{Z}^+\times \mathbb{Z}$, often denoted by unreduced fractions $\frac{n}{d}.$ The number of components of this multicurve is equal to $r=gcd(d,n)$ and the reduced fraction $\frac{n/r}{d/r}$ is called the slope of the multicurve. See Figure \ref{Prod_colors} for a graphic representation of the product $(1,0)*(0,1)*(1,1)$ in the torus. By convention, we denote  the curve $(-d,n)$ by $(d,-n)$, $d>0$. Further, let $(0,0)$ represent the empty link.
	
\begin{figure}
	\centering
	\includegraphics[scale = 1.75]{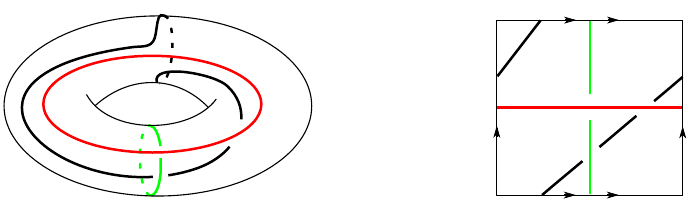}
	\caption{\small{The product $(1,0)*(0,1)*(1,1)$ in the KBSA of $F_{1,0} \times I$.}} 
	\label{Prod_colors}
\end{figure}

Similarly when $F = F_{0,4}$, we define the pair $(d,n) \in \mathbb Z^+ \times \mathbb Z$ (also denoted by $\frac{n}{d}$) to be the multicurve whose (minimal) intersection numbers with the $x$-axis and the $y$-axis are $2|n|$ and $2d$, respectively. 
	The sign of $n$ is given by the ``direction'' of the curve. This is a special case of Dehn coordinates of multicurves 
	in any oriented surface with negative Euler characteristic (see \cite{Dehn1,P-H} for a detailed description).

	\begin{figure}
		\centering
		\includegraphics[width=12.5cm]{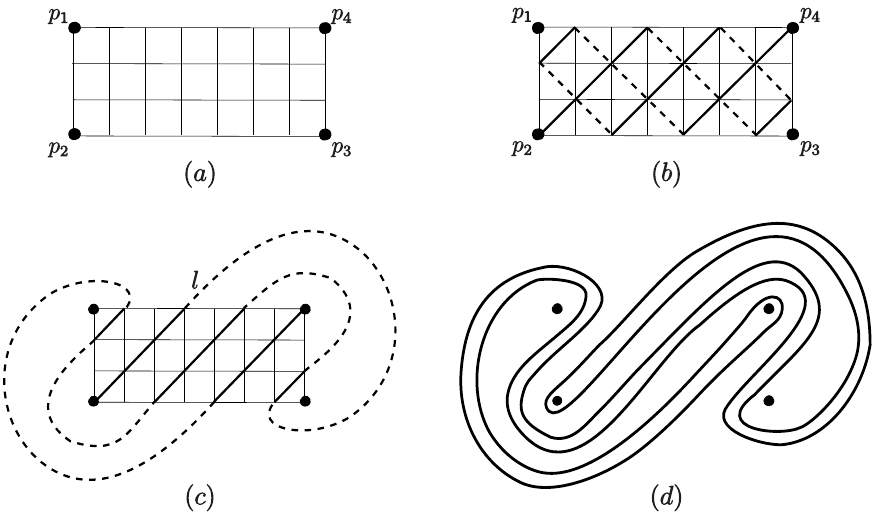}
		\caption{The curve $(3,7) =\frac{7}{3}  $ is drawn  as follows.  Start with a $3\times 7$ grid as shown in (a). The vertices of the rectangle represent the $4$ punctures in $S^2$. We draw an arc starting at the lower left corner of the `pillowcase' in (b). The arc in (c) is the core of the curve $(3,7)$ shown in (d).}
		\label{pillowcase}
	\end{figure}

	\begin{remark}\label{b}
	 There are many ways of drawing curves of given slopes in $F_{0,4}$. Figure \ref{pillowcase} illustrates one such way of drawing the curve $(d,n), \ n > 0,$ in $F_{0,4} \times I$. If $n<0,$ one reflects the curve $(d,|n|)$ about the vertical axis.
	\end{remark}
	
	A presentation of the algebras $S_{2,\infty}(F_{1,0} \times I)$ and $S_{2,\infty}(F_{0,4} \times I)$ is given in \cite{BP}. We now recall this presentation for convenience, since it will be useful when comparing the multiplicative structure of $S_{2,\infty}(F_{0,4} \times I)$ with that of $S_{2,\infty}(F_{1,0} \times I)$. 
	
	The KBSA of $F_{1,0} \times I$ is generated by the simple curves $(0,1), (1,0), (1,1)$; see \cite{BP}. In order to get the set of relations, we start by computing the relations for the KBSA of $F_{1,1} \times I$. Figure \ref{FigF11} illustrates the following products of curves:
	$$
	(1,0) * (0,1) = A (1,1) + A^{-1} (1,-1), \quad \quad (0,1) * (1,0) = A (1,-1) + A^{-1} (1,1).
	$$
	
	Therefore,  
	\begin{equation}\label{eq1}
	A (1,0) * (0,1) - A^{-1} (0,1) * (1,0) = (A^2 - A^{-2}) (1,1).
	\end{equation}
	
	\begin{figure}
		\centering
		\includegraphics[width = 12.5cm]{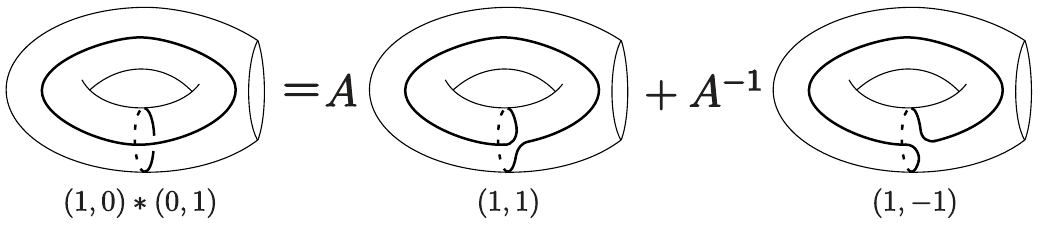}
		\caption{\small{The product $(1,0)*(0,1)$ in the KBSA of $F_{1,1} \times I $.}} 
		\label{FigF11}
	\end{figure}
	
	In a similar way one gets:
	\begin{equation}\label{eq2}
	A (0,1) * (1,1) - A^{-1} (1,1) * (0,1) = (A^2 - A^{-2}) (1,0), 
	\end{equation}
	\begin{equation}\label{eq3}
	A (1,1) * (1,0) - A^{-1} (1,0) * (1,1) = (A^2 - A^{-2}) (0,1).
	\end{equation} 
	
	Equations \ref{eq1}-\ref{eq3} constitute the relations of the KBSA of the punctured torus $F_{1,1}$ (see Theorem \ref{TeoF11F10}(1)). The curve $\partial$, parallel to the boundary, is generated by the former three generators, that is,  
	$$
	(1,-1) * (1,1) = A^2 (1,0)^2 + \partial + d + A^{-2} (0,1)^2,
	$$
	where $d = -A^{2} - A^{-2}$ corresponds to the trivial contractible curve.

	Equations \ref{eq1}-\ref{eq3} also hold for the KBSA of the torus. Moreover, since the boundary curve $\partial$ becomes contractible in $F_{1,0}$, we get the `\textit{long relation}' for the KBSA of the torus (the $\ell^{th}$-power of a curve denotes $\ell$ parallel copies of it):
	\begin{equation}\label{eq4}
	A^2 [(1,1)^2 + (1,0)^2] + A^{-2} (0,1)^2 - 2(A^{2} + A^{-2}) - A ((1,0) * (0,1) * (1,1)) = 0.
	\end{equation}
	
	\begin{theorem}\cite{BP}\label{TeoF11F10}
		As an algebra, 
		\begin{enumerate}
			\item $\mathcal{S}_{2,\infty}(F_{1,1} \times I) = k \left\{(1,0), (0,1), (1,1)\right\} / (\ref{eq1}), (\ref{eq2}), (\ref{eq3})$.
			\item $\mathcal{S}_{2,\infty}(F_{1,0} \times I) = k \left\{(1,0), (0,1), (1,1)\right\} / (\ref{eq1}), (\ref{eq2}), (\ref{eq3}), (\ref{eq4})$.
		\end{enumerate}
	\end{theorem}
	
	Now, we consider the KBSA of $F_{0,4}\times I$, which is generated (as an algebra) by the curves $(1,0), (0,1), (1,1)$ together with the curves $a_i$ where $i=1,2,3,4$ and the curves $a_{i}$  are parallel to $p_i$ (see Figure \ref{pillowcase} (a)). Recall that the boundary curves generate the center of the algebra. We can consider a bigger commutative ring $K = k[a_1, a_2, a_3, a_4]$ and consider the KBSA over this ring. The computation of the product of  the curves $(1,0)$ and $(0,1)$ is illustrated in Figure \ref{FigF10}:
	$$
	(1,0) * (0,1) = A^2 (1,1) + a_1a_3 +a_2a_4 + A^{-2}(1,-1).$$ Therefore, 
	$$
	(0,1) * (1,0) = A^2 (1,-1) + a_1a_3 + a_2a_4 + A^{-2} (1,1).$$
	
	\begin{figure}
		\centering
		\includegraphics[width = 13.1cm]{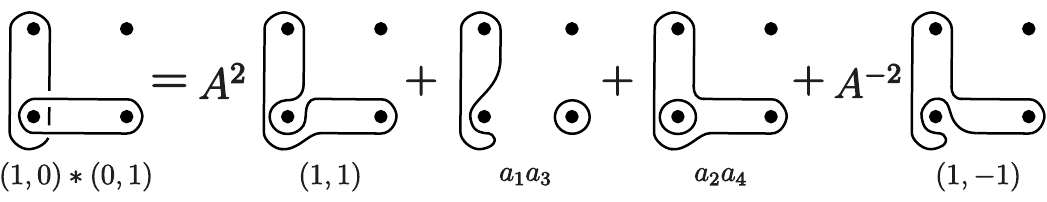}
		\caption{\small{The product $(1,0)*(0,1)$ in $F_{0,4}$.}} 
		\label{FigF10}
	\end{figure}
	
	\begin{figure}
		\centering
		\includegraphics[width = 12.5cm]{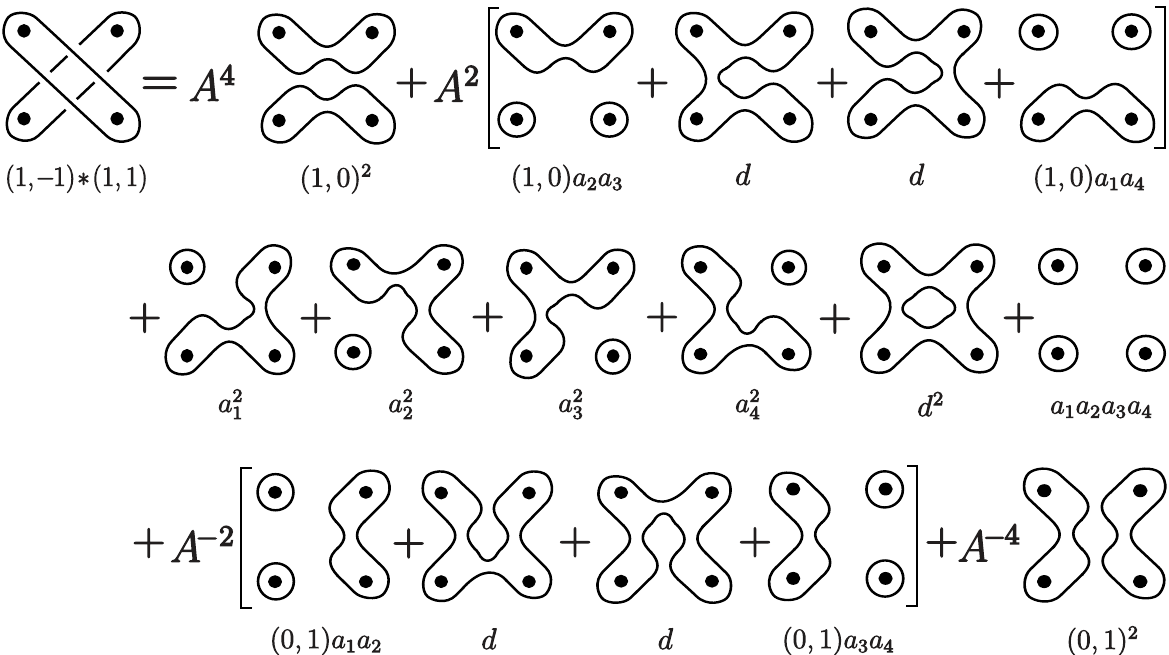}
		\caption{\small{The product $(-1,1)*(1,1)$ in $F_{0,4}$.}} 
		\label{FigF04}
	\end{figure}
	
	From these equations we obtain the following:
	\begin{equation}\label{eq11}
	A^2 (1,0) * (0,1) - A^{-2} (0,1) * (1,0) = (A^4 - A^{-4}) (1,1)+ (A^2 - A^{-2}) (a_1a_3+a_2a_4),
	\end{equation} 
	\begin{equation}\label{eq22}
	A^2 (0,1) * (1,1) - A^{-2} (1,1) * (0,1) = (A^4 - A^{-4}) (1,0)+ (A^2 - A^{-2}) (a_1a_4+a_2a_3),
	\end{equation} 
	\begin{equation}\label{eq33}
	A^2 (1,1) * (1,0) - A^{-2} (1,0) * (1,1) = (A^4 - A^{-4}) (0,1)+ (A^2 - A^{-2}) (a_1a_2+a_3a_4).
	\end{equation} 
	
	 The {\it long relation} for $F_{0,4}$ is
	\begin{eqnarray}\label{eq44}
	A^4 \left[(1,1)^2 + A^2 (1,0)^2\right] + A^{-4} (0,1)^2 + A^2 \left[(a_1a_4+a_2a_3)(1,0) + (a_1a_3+a_2a_4) (1,1)\right] +  \\
	A^{-2}(a_1a_2+a_3a_4)(0,1) + \left[a_1^2+a_2^2+a_3^2+a_4^2 + a_1a_2a_3a_4 + (A^2-A^{-2})^2\right] -2(A^4+A^{-4}) - \nonumber \\  A^2 (1,0) * (0,1) * (1,1) = 0. \nonumber
	\end{eqnarray}
	
	\begin{theorem}\cite{BP}\label{TeoF04}
		As an algebra, 
		$\mathcal{S}_{2,\infty}(F_{0,4} \times I) = k[a_1,a_2,a_3,a_4] \left\{(1,0), (0,1), (1,1)\right\} / (\ref{eq11}), (\ref{eq22}), (\ref{eq33}), (\ref{eq44}).$
	\end{theorem}
	
	\begin{corollary}\label{2.8}\cite{BP}
		Let $R_{0,1}= a_1a_2+a_3a_4$, $R_{1,0}=a_1a_4+a_2a_3$, $R_{1,1}=a_1a_3+a_2a_4$ and $y = a_1^2 +a_2^2 +a_3^2 +a_4^2 +a_1a_2a_3a_4 + (A^{2} - A^{-2})^2$. 
		Then, modulo $R_{0,1}, R_{1,0}, R_{1,1}$ and $y$, the relations associated to the KBSA of $F_{1,0} \times I$ and $F_{0,4} \times I$  coincide upto switching $A$ and $A^2$ and keeping the slopes of the curves unchanged.
	\end{corollary}
	
	 Corollary \ref{2.8} suggests that when looking for formulas for the multiplication of two curves in $F_{0,4}$ based on the \textit{product-to-sum} formula described in \cite{FG} (see Theorem \ref{TeoProdtoSum}), one should work with a subring of the center of  $\mathcal{S}_{2,\infty}(F_{0,4} \times I)$ generated by the 
	variables $R_{0,1}, R_{1,0}, R_{1,1}$ and $y$. We denote this subring by $K_{0}$.

	\subsection{Chebyshev polynomials and the product-to-sum formula by Frohman and Gelca}
	
	In \cite{FG} Frohman and Gelca introduced a formula for the multiplication of two elements in the KBSA of the torus. Their formula is expressed in terms of Chebyshev polynomials of the first kind. 
	
	The Chebyshev polynomials of the first kind are defined by the recurrence relation $$ T_{n}(x) = xT_{n-1}(x) - T_{n-2}(x)$$ where  $T_{0}(x) = 2$, $T_{1}(x) = x$ are the initial conditions. 
	
	Chebyshev polynomials of the second kind, $S_n(x)$, are defined by the same recurrence relation together with the initial conditions $S_{0}(x) = 1$ and $S_1(x) = x$.
	
	The following straighforward properties will be useful in the next sections:
	
	\begin{proposition}\cite{Lic, Chebyshev}\label{property_cheby} \
		\begin{enumerate}
			\item If we write $x = a + a^{-1}$, then $T_{n}(x)  = a^{n} + a^{-n}$ and $S_{n}(x) = \frac{a^{n+1} - a^{-n-1}}{a - a^{-1}} $.
			\item  The product-to-sum formula for Chebyshev polynomials is given by: $$T_m(x)T_n(x)= T_{m+n}(x) + T_{|m-n|}(x).$$
			\item $T_n(x) = S_n(x) - S_{n-2}(x)$.
			\item $S_n(x) = \bigg( \displaystyle \sum_{i=0}^{\lfloor (n-1)/2\rfloor}T_{n-2i}\bigg)+ \alpha(n)$, where $\alpha(n) = 0$ if $n$ is odd and $1$ otherwise.
		\end{enumerate}
	\end{proposition}\

	Given $(d,n),$ a multicurve in $F_{1,0} \times I$ with $r = gcd(d,n)$, let
	$$(d,n)_{T}=T_r((d/r,n/r)) \quad \mbox{ and } \quad (d,n)_{S}=S_{r}((d/r,n/r)).$$ 
	The above expressions are elements in $S_{2, \infty}(F_{1,0} \times I)$ and denote formal linear combinations of muticurves on the torus where the variable of Chebyshev polynomials is substituted by the curves on the torus.
	
	\begin{example} \
		\begin{enumerate} 
			\item $(8,2)_T = T_2((4,1)) = (4,1)^2-2 = (8,2)-2,$ where $2$ denotes $2(0,0)$. 
			\item $(3,6)_S = S_3((1,2)) = (1,2) * [(1,2)^2-1] - (1,2) = (3,6) - 2 (1,2)$. 
		\end{enumerate}
	
	\end{example}
	
	\begin{theorem}\cite{FG} (Product-to-sum formula)\label{TeoProdtoSum}
		Given two multicurves $(d_1,n_1), (d_2,n_2) \in S_{2, \infty}(F_{1,0} \times I)$, 
		$$(d_1,n_1)_{T}*(d_2,n_2)_{T}=A^{d_1n_2-n_1d_2}(d_1+d_2,n_1+n_2)_{T}+A^{-(d_1n_2-n_1d_2)}(d_1-d_2,n_1- n_2)_{T}.$$
	\end{theorem}


\section{Product-to-sum formula for some families of curves in the thickened T-shirt} \label{SecFamilies}

\subsection{Basic formulas and Initial Data}\

Given the product of two multicurves $(d_1,n_1)*(d_2,n_2)$, we associate to it the {\it determinant} $(d_1n_2-d_2n_1)$. In this subsection we provide some basic computation results for the multiplication of two multicurves when the absolute value of their associated determinant is either $0$, $1$ or $2$. These will serve as the initial data for the algorithm presented in Section \ref{SecAlgorithm}.
 
As in the case of the torus, we denote by $(d,n)_T$ the multicurve $(d,n)$ decorated by the Chebyshev polynomial $T_{gcd(d,n)}\left(\frac{d}{gcd(d,n)}, \frac{n}{gcd(d,n)}\right)$. Consider the product $(d_{1},n_{1})*(d_{2},n_{2})$ with determinant $d_{1}n_{2} - d_{2}n_{1} = 0$. We have $(d_{1},n_{1})*(d_{2},n_{2}) = (d_{1} + d_{2}, n_{1} + n_{2})$. From the product-to-sum formula for Chebyshev polynomials, Proposition \ref{property_cheby} (2), we have $(d_{1},n_{1})_{T}*(d_{2},n_{2})_{T} = (d_{1} + d_{2}, n_{1} + n_{2})_{T} + (d_{1} - d_{2}, n_{1} - n_{2})_{T}$. 

\begin{lemma}\label{det1}
Consider two curves in $F_{0,4},$ $(d_1,n_1)=\frac{n_1}{d_1}$ and $(d_2,n_2)=\frac{n_2}{d_2},$ satisfying $d_1n_2-d_2n_1=1$. Then
$$(d_1,n_1)*(d_2,n_2) = A^2(d_{1}+d_{2},n_{1}+n_{2}) + A^{-2}(d_{1}-d_{2},n_{1}-n_{2}) +R_{d_1+d_2,n_1+n_2},$$
where the indices of $R$ are taken modulo $2$. 

In particular we have: 

\begin{enumerate}
\item[(i)]$(1,0)*(1,1)= A^2(2,1) + A^{-2}(0,1) + a_1a_2+a_3a_4 = A^2(2,1) + A^{-2}(0,1) + R_{0,1},$

\item [(ii)]$(1,0)*(0,1)= A^2(1,1) + A^{-2}(-1,1) + a_1a_3+a_2a_4 = A^2(1,1) + A^{-2}(-1,1) + R_{1,1},$
\item [(iii)]
$(1,1)*(0,1)= A^2(1,2) + A^{-2}(1,0) + a_1a_4+a_2a_3 = A^2(1,2) + A^{-2}(1,0) + R_{1,0}.$	

\end{enumerate}
\end{lemma}

\begin{proof}
It suffices to	prove one of cases (i), (ii) or (iii). The lemma follows by applying the action of $SL(2,\mathbb{Z})$ on the basic elements in the formulas (in Section \ref{SecAlgorithm} we discuss the effects of this action; compare \cite{FM}). The calculation for case (ii) is presented in Figure \ref{FigF10}.
\end{proof}

\begin{lemma}\label{det21}
Let $(d_1,n_1), (d_2,n_2)$ be two curves in $F_{0,4}$ so that $gcd(d_1,n_1)=gcd(d_2,n_2)=1$ and $d_1n_2-d_2n_1=2$. Then
	$$(d_1,n_1)*(d_2,n_2)= A^{4}(d_{1}+d_{2},n_{1}+n_{2})_T + A^{-4}(d_{1}-d_{2},n_{1}-n_{2})_T + y +$$
	$$A^{2}R_{\frac{d_1+d_2}{2},\frac{n_1+n_2}{2}}((d_1+d_2)/2,(n_1+n_2)/2) + 
	A^{-2}R_{\frac{d_1-d_2}{2}\frac{n_1-n_2}{2}}((d_1-d_2)/2,(n_1-n_2)/2).$$
	
	In particular, we have $(1,-1)*(1,1) = A^{4}(2,0)_T + A^{-4}(0,2)_T + y + A^{2}R_{1,0}(1,0) + A^{-2}R_{0,1}(0,1)$. 
\end{lemma}

\begin{proof}
Similarly, as in the previous lemma, it suffices to find the formula for $(1,-1)*(1,1)$. Calculations for this case are shown in Figure \ref{FigF04}.
\end{proof}	

\begin{lemma}\label{det22}
Let $(d_1,n_1), (d_2,n_2)$ be two multicurves in $F_{0,4}$ so that $d_1n_2-d_2n_1=2$. \
	
\begin{enumerate}
	\item If $gcd(d_1,n_1)=2$, then  $$(d_1,n_1)_T*(d_2,n_2)= A^{4}(d_{1}+d_{2},n_{1}+n_{2}) + A^{-4}(d_{1}-d_{2},n_{1}-n_{2}) +$$ $$(d_{1}/2,n_{1}/2)R_{d_1/2 +d_2,n_1/2+n_2} + (A^2+A^{-2})R_{d_2,n_2}.$$ 
	
	\item If $gcd(d_2,n_2)=2$, then
	$$(d_1,n_1)*(d_2,n_2)_T= A^{4}(d_{1}+d_{2},n_{1}+n_{2}) + A^{-4}(d_{1}-d_{2},n_{1}-n_{2}) + $$
	$$R_{d_1 +d_2/2,n_1+n_2/2}(d_{2}/2,n_{2}/2) + (A^2+A^{-2})R_{d_1,n_1}.$$
\end{enumerate}
\end{lemma}

\begin{proof}
	
We prove case $(1)$ by applying Lemma \ref{det1} twice:	

$$(d_1,n_1)_T*(d_2,n_2)=\bigg(\frac{n_1}{d_1}\bigg)_T * \frac{n_2}{d_2} = \frac{n_1/2}{d_1/2}*\bigg(\frac{n_1/2}{d_1/2}*\frac{n_2}{d_2}\bigg) - 2\bigg(\frac{n_2}{d_2}\bigg) \stackrel{Lemma \ref {det1}}{=\joinrel=\joinrel=\joinrel=\joinrel=\joinrel=}$$
$$ \frac{n_1/2}{d_1/2}*\left(A^2\frac{n_1/2 +n_2}{d_1/2+d_2} + A^{-2}\frac{n_2-n_1/2}{d_2 - d_1/2} +R_{d_1/2 +d_2,n_1/2+n_2}\right) - 2\bigg(\frac{n_2}{d_2}\bigg) \stackrel{Lemma \ref {det1}}{=\joinrel=\joinrel=\joinrel=\joinrel=\joinrel=} $$
$$A^2\left(A^2\frac{n_1+n_2}{d_1+d_2} + A^{-2}\frac{n_2}{d_2}+ R_{d_2,n_2}\right) + $$ 
$$A^{-2}\left(A^2\frac{n_2}{d_2} +A^{-2}\frac{n_2-n_1}{d_2-d_1} + R_{d_2,n_2}\right) + \frac{n_1/2}{d_1/2}R_{d_1/2 +d_2,n_1/2+n_2} - 2\bigg(\frac{n_2}{d_2}\bigg)=$$
$$A^4\frac{n_1+n_2}{d_1+d_2} + A^{-4}\frac{n_2-n_1}{d_2-d_1}+ 
 \frac{n_1/2}{d_1/2}R_{d_1/2 +d_2,n_1/2+n_2} + (A^2+A^{-2})R_{d_2,n_2} .$$
 
 The proof of case $(2)$ is analogous. 	
\end{proof}	

\begin{remark}
After switching $A$ by $A^{-1}$ Lemmas \ref{det1} holds when the determinant of the curves is equal to $-1$ and Lemmas \ref{det21} and \ref{det22} hold when the determinant of the curves is equal to $-2$. 
\end{remark}

\subsection{Two closed formulas}
In this subsection we present formulas for the product of two infinite families of curves and the curve $(0,1)$.

\begin{theorem}\label{Lemma 14.1} Let $m>0$. Then, 
	$$(m,0)_T*(0,1) = A^{2m}(m,1) + A^{-2m}(m,-1) + (m-1,0)_SR_{1,1} +$$
	$$\sum_{i=1}^{m-1}(A^{2i}+A^{-2i})(m-1-i,0)_SR_{i+1,1}. $$
\end{theorem}

\begin{proof} For $m=1,2$ the formula follows from Lemmas \ref{det1} and \ref{det22}. Now we assume that the formula holds up to $m$ ($m>2$) and proceed by induction on $m$ by applying the recurrence relation of Chebyshev polynomials:
	
	$$(m+1,0)_T*(0,1) \stackrel{Chebyshev}{=\joinrel=\joinrel=\joinrel=\joinrel=\joinrel=\joinrel=\joinrel} (1,0)*(m,0)_T*(0,1)- (m-1,0)_T*(0,1)\stackrel{Induction}{=\joinrel=\joinrel=\joinrel=\joinrel=\joinrel=\joinrel=\joinrel} $$
	$$(1,0)*[A^{2m}(m,1) + A^{-2m}(m,-1)]- [A^{2m-2}(m-1,1) + A^{-2m+2}(m-1,-1)] + $$ 	
	$$(1,0)*(m-1,0)_SR_{1,1} - (m-2,0)_SR_{1,1} + $$
	$$(1,0)*\left[\sum_{i=1}^{m-1}(A^{2i}+A^{-2i})(m-1-i,0)_SR_{i+1,1}\right] 
	- \sum_{i=1}^{m-2}(A^{2i}+A^{-2i})(m-2-i,0)_SR_{i+1,1} \stackrel[Det=\pm 1]{Chebyshev}{=\joinrel=\joinrel=\joinrel=\joinrel=\joinrel=\joinrel=\joinrel} $$
	$$A^{2m+2}(m+1,1)+A^{2m-2}(m-1,1)+ A^{2m}R_{m+1,1}+ A^{-2m}R_{m+1,1}+ A^{2-2m}(m-1,-1) $$
	$$+ A^{-2m-2}(m+1,-1) + -A^{2m-2}(m-1,1) - A^{2-2m}(m-1,-1) + $$ 
	$$(m,0)_SR_{1,1} +
	\sum_{i=1}^{m-1}(A^{2i}+A^{-2i})(m-i,0)_SR_{i+1,1}=$$
	$$A^{2m+2}(m+1,1)+ A^{-2m-2}(m+1,-1) + (m,0)_SR_{1,1} +$$
	$$\sum_{i=1}^{m}(A^{2i}+A^{-2i})(m-i,0)_SR_{i+1,1},$$ as needed.\footnote{Notice that we use the fact that the Chebyshev polynomials $T_n$ and $S_n$ have the same recursive relation.}
\end{proof}

Let $[n]_{q}=1+q+q^2+...+q^{n-1}$ denote the $n^{th}$ quantum integer. The following theorem, whose proof is given in the Appendix, gives the formula  for the product of curves in the family $(n,1)*(0,1).$

\newtheorem*{thm:xiao}{Theorem \ref{thm:xiao}}
\begin{theorem}\label{thm:xiao}
	
	Let $n$ be a non-negative integer. Then, 
	$$(n,1)*(0,1)=A^{2n}(n,2)_T + A^{-2n}(n,0)_T + \sum_{i=0}^{n-1}\alpha_{n-1}\beta_i +A^2\sum_{i=1}^{n-1}[n_i]_{A^4} \frac{1}{i}R_{i-1,1},$$
	where
	\[
	n_i=
	\left\{
	\begin{array}{ll}
	\left[i\right]_{A^4} &  \mbox{for $ i\leq \frac{n}{2}, $}\\
	\left[n-i\right]_{A^4} & \mbox{for $i\geq \frac{n}{2},$}
	\end{array}
	\right.
	\]

	\begin{minipage}[t]{0.45\textwidth}
		$$ \alpha_{i}=\left\{
		\begin{array}{rcl}
		R_{1,0}       &      & {i=1,}\\
		y     &      & {i=2,}\\
		R_{0,1}R_{1,1}+R_{1,0}     &      & {i=3,}\\
		\frac{(i-2)}{2}(R_{0,1}^{2}+R_{1,1}^{2})       &      & {i>3  \mbox{ even,}}\\
		(i-2)R_{0,1}R_{1,1}       &      & {i>3 \mbox{ odd, }} 
		\end{array} \right. $$
	\end{minipage}
	\begin{minipage}[t]{0.45\textwidth}
		$$\mbox {and } \ \beta_{i}=\left\{
		\begin{array}{lll}
		\sum\limits_{j=0}^{(i-1)/2} A^{-2i+8j}(i-2j,0)_{S}      &      & {i  \mbox{ odd,}}\\
		\sum\limits_{j=0}^{i/2} A^{-2i+8j}(i-2j,0)_{S}        &      & {i \mbox{ even}}.
		\end{array} \right. $$
	\end{minipage}
	\ \\
	
	Notice that all the coefficients are positive and unimodal.  
	
\end{theorem}

\section{The algorithm} \label{SecAlgorithm}

In this section we describe how to compute the product of closed multicurves in $F_{0,4}\times I$ algebraically. 

Recall that the mapping class group of the torus is $Mod(T^{2}) = SL(2,\mathbb Z)($\cite{Dehn1}), and that of $F_{0,4}$ is the semidirect product: $Mod(F_{0,4}) = PSL(2,\mathbb Z) \ltimes (\mathbb Z_2 \times \mathbb Z_2)$ (we allow permutations of the boundary components, compare \cite{FM}). The action of $PSL(2,\mathbb Z)$ on $\mathbb Q \cup \frac{1}{0}$ is encoded by the Farey Diagram shown in Figure \ref{fig:farey-sujoy}. Following is a short description (see \cite{Hat} for more detail). 

The Farey diagram consists of vertices indexed by fractions $\mathbb Q \cup \frac{1}{0},$ where vertices $\frac{a}{b}$ and $\frac{c}{d}$ are connected by an edge if the determinant $(ad - bc)$ of the matrix  
$\left[ {\begin{array}{cc}
	a & c \\
	b & d \\
	\end{array} } \right]$ is $\pm 1$.
		The vertices of the Farey diagram are labeled according to the following scheme: if the labels at the ends of the long edge of a triangle are $\frac{a}{b}$ and $\frac{c}{d}$, then the label on the third vertex of the triangle is the fraction $\frac{a+b}{c+d}$. This fraction is known as the \textit{mediant} of $\frac{a}{b}$ and $\frac{c}{d}$.
	
	$PSL(2,\mathbb{Z})$ acts on $F_{0,4}$ as follows: the generators $s_{1}=\left[ {\begin{array}{cc}
		1 & 1 \\
		0 & 1 \\
		\end{array} } \right]$ and $s_{2}=\left[ {\begin{array}{cc}
		1 & 0 \\
		1 & 1 \\
		\end{array} } \right]$ act on the fractions $\frac{n}{d}$ by matrix multiplication.  The generators act on the center by switching $a_3$ with $a_4$ and $a_1$ with $a_4$ respectively. Therefore, under the action of $s_1$, $R_{1,1}$ goes to $R_{1,0}$, $R_{0,1}$ goes to $R_{1,1}$ and $R_{0,1}$ remains unchanged. The second generator, $s_2$, sends $R_{1,1}$ to $R_{0,1}$, $R_{0,1}$ to $R_{1,1}$ and keeps $R_{1,0}$ fixed. The variable $y$ remains unchanged by the action.
 
\begin{figure}
	\centering
		\includegraphics[width=0.8\linewidth]{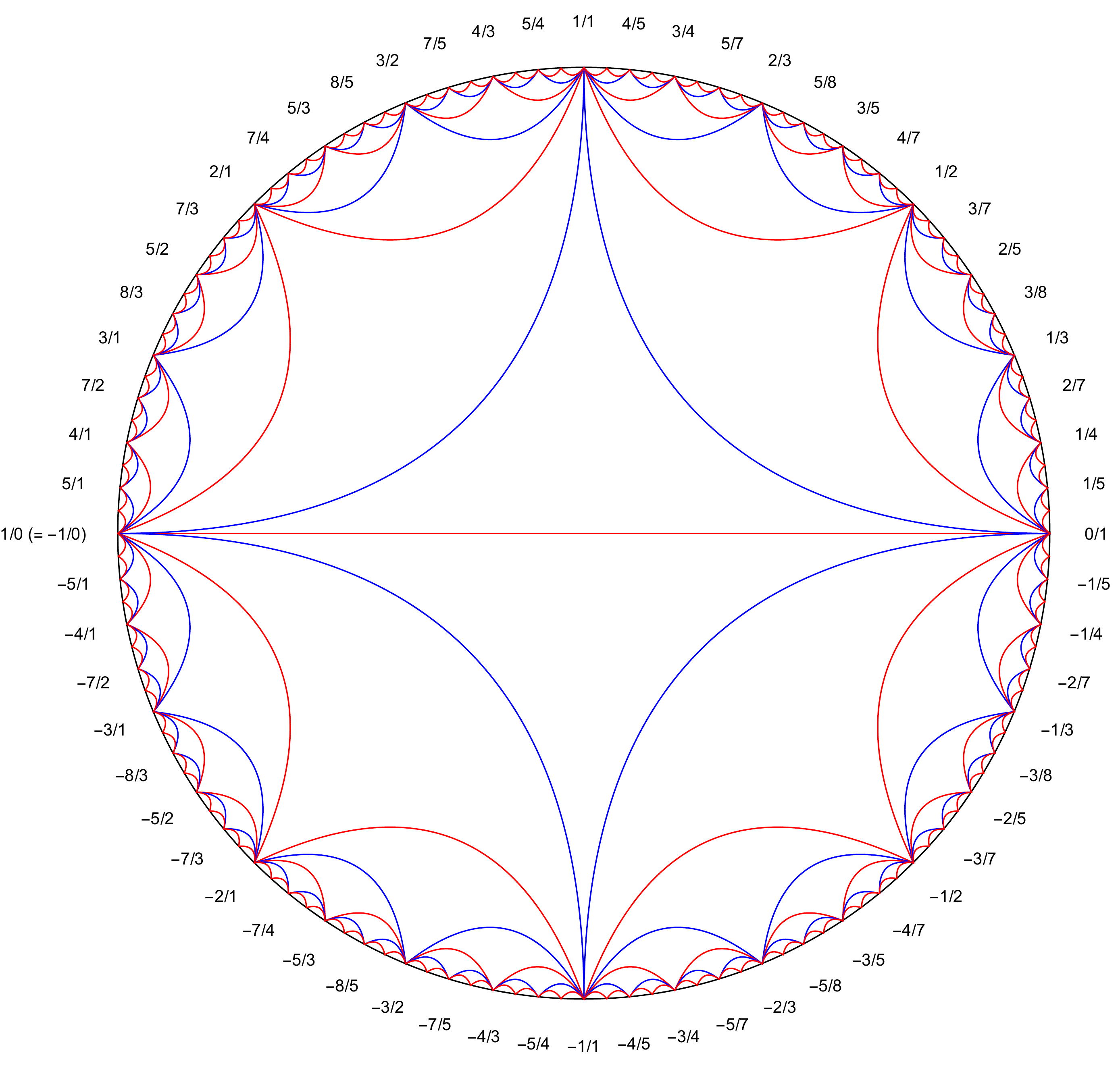}
	\caption{The Farey diagram.}
	\label{fig:farey-sujoy}
\end{figure}

We now analyze the algorithm which, given an ordered pair of curves in $F_{0,4} \times I$, returns their product as an element of KBSA$(F_{0,4} \times I)$ in terms of the generators in the basis. The first step in the algorithm consists of reducing the product $(d_1, n_1)*(d_2, n_2)$ to the form $(d,n)*(0,k),$ with $0\leq n<d$ and $k>0$. This is achieved by using the action of $PSL(2,\mathbb{Z})$ on $F_{0,4}$ together with an adapted version of the Euclidean algorithm. 

We now continue with the following lemmas in preparation for step 2 of the algorithm. 

Lemma \ref{lemma 1} relates the determinant ($d_1n_2 - d_2n_1$) of $\left[ {\begin{array}{cc}
	d_1 & n_1 \\
	d_2 & n_2 \\
	\end{array} } \right]$ with the geometric intersection number of a pair of multicurves $(d_1, n_1)$ and $(d_2, n_2)$. Lemma \ref{multi} is 
used to show the correctness of the algorithm when $k >1$ or $gcd(d,n)>1$.  Lemmas \ref{Lemma 4.4} and \ref{Lemma 4.5} are used in Theorem \ref{main} to prove the correctness of the algorithm when $k=1$ and $gcd(d,n)=1$. 

\begin{lemma}\label{lemma 1}
	
	Consider any two simple closed multicurves $(c_{1}d_{1},c_{1}n_{1})$ and $(c_{2}d_{2},c_{2}n_{2})$ with $gcd(d_{1},n_{1})=1=gcd(d_{2},n_{2})$ in $F_{0,4} \times I$, where $c_{i}, d_{i}, n_{i} \in \mathbb{Z}$, and $i = 1, 2$. Then, their geometric intersection number is equal to $2|c_{1}d_{1}c_{2}n_{2}-c_{1}n_{1}c_{2}d_{2}|$. See section 2.2.5 in \cite{FM}.	
\end{lemma}

\begin{lemma}\label{multi}
	
	Consider $(c_{1}d_{1},c_{1}n_{1})*(c_{2}d_{2},c_{2}n_{2})$ with $gcd(d_{1},n_{1})=1=gcd(d_{2},n_{2})$. Without loss of generality, assume $c_{1}>c_{2}\geq1$.  Suppose $((c_{1}-1)d_{1},(c_{1}-1)n_{1})*(c_{2}d_{2},c_{2}n_{2})=\sum_{i}u_{i}(q_{i},p_{i})$, $u_{i} \in Z(\mathcal{S}_{2,\infty}(F_{0,4}))$, then we have $|d_{1}p_{i}-n_{1}q_{i}|<|c_{1}d_{1}c_{2}n_{2}-c_{1}n_{1}c_{2}d_{2}|$.
	
\end{lemma}

\begin{proof}
	
	 The curves $(q_{i},p_{i})$ are obtained after one smooths all the crossings of $(c_{1}-1)$ curves $(d_{1},n_{1})$ and $c_{2}$ curves $(d_{2},n_{2})$.  During this process, we observe that the number of crossings, say $n$, of the remaining copy of the curve $(d_{1},n_{1})$ with $(q_{i},p_{i})$ does not increase. Thus, $n=2|d_{1}p_{i}-n_{1}q_{i}| \leq 2|d_{1}c_{2}n_{2}-n_{1}c_{2}d_{2}|\leq|c_{1}d_{1}c_{2}n_{2}-c_{1}n_{1}c_{2}d_{2}|<2|c_{1}d_{1}c_{2}n_{2}-c_{1}n_{1}c_{2}d_{2}|$.	
\end{proof}

\begin{lemma}\label{Lemma 4.4}
	Consider a curve  $(d,n)$ ($d\geq 0$) and let $(d,n)*(0,1)=\sum_i u_i(q_{i},p_{i})$,  
	with $q_i\geq 0$. Then,
	\begin{enumerate}
		\item[(1)]$|p_i|\leq n+1$ and $q_i\leq d$.
		\item[(2)]$|p_id-q_in|\leq d$.
		\item[(3)]If $n>0$ then $p_i\geq 0$. 
	\end{enumerate}
\end{lemma}
\begin{proof}\ 
	
\begin{enumerate}	
\item{It follows from the observation that the coordinate lines cannot cut a $(q_i,p_i)$ curve more than the number of times the curves
$(d,n)$ and $(0,1)$ together cut the $(q_i,p_i)$ curve.}
\item{We notice that the curve $(d,n)$ can intersect any curve of $(d,n)*(0,1)$ (including $(q_{i},p_{i})$) in no more than $2d$ points.  Thus, by Lemma \ref{lemma 1}, the statement follows.}
\item{This case follows from (2) since we have $-d \leq p_id-q_in$ so $\frac{n}{d}q_i-1 \leq p_i,$ and therefore for $n>0$ and $q_i>0$ we have 
	$p_i>-1.$ If $q_i=0$ then we can assume $p_i\geq 0$ by convention.}
\end{enumerate}  
\end{proof}

Lemma \ref{Lemma 4.4} is generalized in Theorem \ref{theorem 15.2}.

\begin{lemma}\label{Lemma 4.5}
	\begin{enumerate}\
	
		\item[(1)] Consider the fractions $\frac{n_1}{d_1}$,$\frac{n_2}{d_2}$, and $\frac{p}{q}$ with positive denominators. Then
		$$|n_2q-d_2p| \leq \frac{q}{d_1}|n_2d_1-n_1d_2| + \frac{d_2}{d_1}|n_1q-d_1p|.$$
		\item[(2)] In particular, if $|n_2d_1-n_1d_2|=1$, $|n_{1}q-d_{1}p|\leq d_{1}$ and $q\leq d_1,$ then $$|n_2q-d_2p| \leq 1+d_2.$$
	\end{enumerate}
\end{lemma}
\begin{proof} Inequality in (1) is essentially triangle inequality for fractions (rational numbers). We see this 
	immediately when we divide two sides of the inequality in (1) by a positive number $qd_2$. We obtain the inequality:
	$$|\frac{n_2}{d_2}- \frac{p}{q}| \leq |\frac{n_2}{d_2}-\frac{n_1}{d_1}| + |\frac{n_1}{d_1}-\frac{p}{q}|$$
	which is just a triangle inequality of rational numbers.
\end{proof}

\begin{figure}
	\centering
	\includegraphics[scale=0.5]{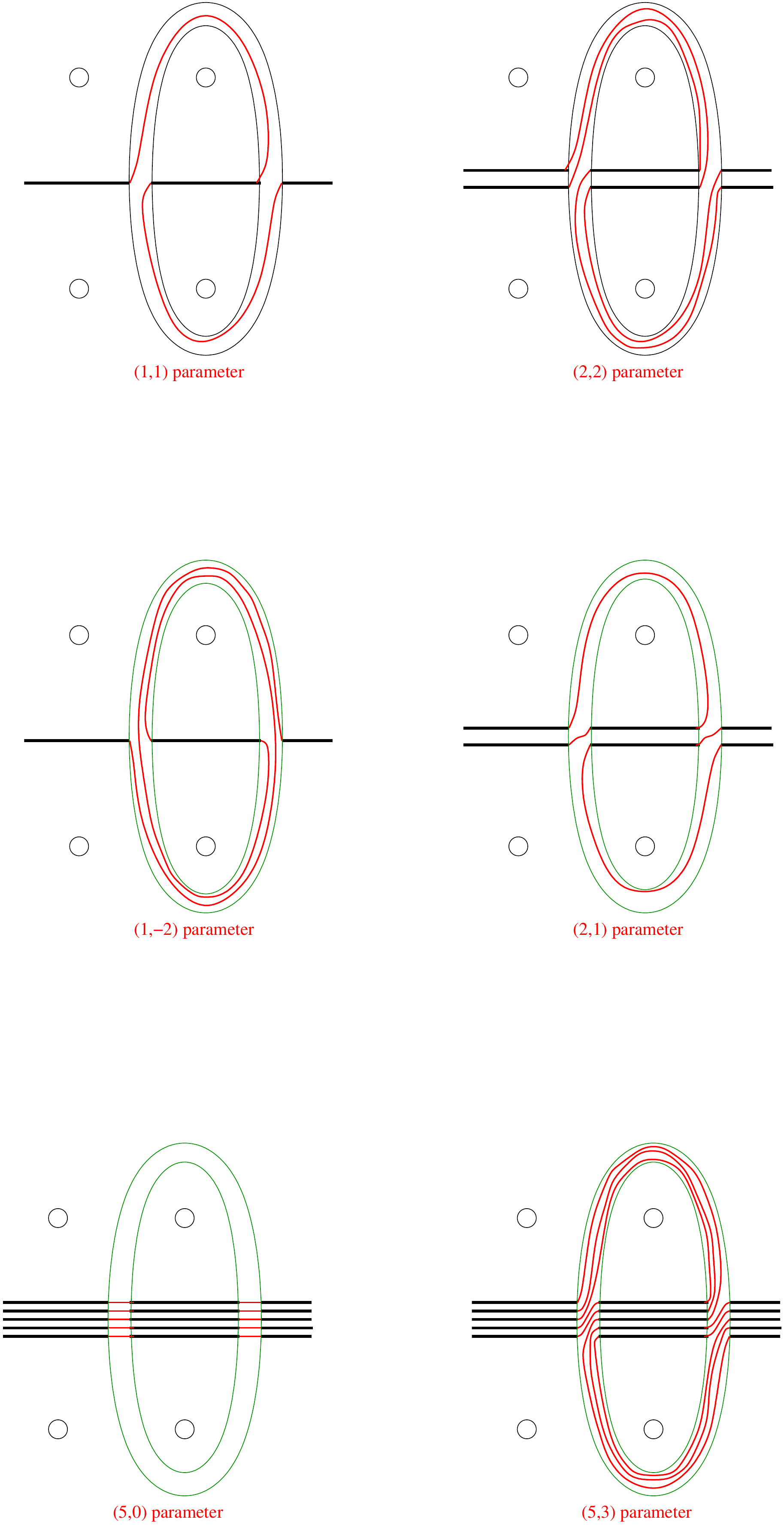}
	\caption{Illustrations of Dehn coordinates.}
	\label{fig:model-curves}
\end{figure}

\begin{theorem}\label{main}

Consider the product $(d,n)*(0,1)$ with $1< n< d \geq 3$ and $gcd(d,n)=1$.  
	Given products of pairs of curves with determinant less than $d$, we can compute $(d,n)*(0,1)$. 

\end{theorem}

\begin{proof}
	
	Let $n,d >1$, $gcd(n,d)=1$, and assume that the reduced fractions $\frac{n_1}{d_1}$ and $\frac{n_2}{d_2}$ ``support"
	$\frac{n}{d}$ in the Farey diagram, that is, $n=n_1+n_2$, $d=d_1+d_2$ and $|n_2d_1-n_1d_2|=1$.
		
		With our assumptions, $d_{1},d_{2}$ and $n_{1},n_{2}$ satisfy $0<d_{2}<d$, $0<d_{1}<d-1$, $0<n_{1},n_{2}$.  Then, $(d,n)*(0,1)=A^{-2(d_{1}n_{2}-d_{2}n_{1})}(d_{1},n_{1})*(d_{2},n_{2})*(0,1)-A^{-2(d_{1}n_{2}-d_{2}n_{1})}R_{d,n}(0,1)-A^{-4(d_{1}z-d_{2}n_{1})}(d_{1}-d_{2},n_{1}-n_{2})*(0,1)$.
For $(d_{1}-d_{2},n_{1}-n_{2})*(0,1)$, we know the product since the absolute value of determinant is $|d_{1}-d_{2}|<d$.  So let us focus on $(d_{1},n_{1})*(d_{2},n_{2})*(0,1)$. Assume $(d_{1},n_{1})*(d_{2},n_{2})*(0,1)=(d_{1},n_{1})*\sum_{i}u_{i}(q_{i},p_{i})$ as in Lemma \ref{Lemma 4.4} (1). By parts (1) and (2) of Lemma \ref{Lemma 4.4} and part (2) of Lemma \ref{Lemma 4.5}, we have $|n_{1}p_{i}-d_{1}p_{i}|\leq1+d_{1}<d.$																																			
\end{proof}

We summarize the above results in the following theorem which proves the correctness of the algorithm.

\begin{theorem}\label{algorithm}
	
	Consider the product $(d,n)*(0,k)$ with $0\leq n\leq d > 0$, $k\geq 0$, one can compute the product algebraically with the knowledge of  products of $(d_{1},n_{1})*(d_{2},n_{2})$ with $|d_{1}n_{2}-d_{2}n_{1}|<dk.$
	
\end{theorem}

\begin{proof}
	If $k=0$, $(d,n)*(0,0)=(d,n)$. If $k>1$, by Lemma \ref{multi}, the statement is true. Now, consider the case $k=1$.
	We can assume that $d=ms+r$, $n=s$, and we multiply $(ms+r,s)*(0,1)$ as before to get
	$$(ms+r,s)*(0,1) = A^{2(ms+r)}(ms+r,s+1) + A^{-2(ms+r)}(ms+r,s-1) + \sum_{i} u_i (q_{i},p_{i}).$$ 
	If $n=1$, we have the closed formula in Theorem \ref{thm:xiao}.\\
	If $n>1$, then\\
	(1) when $gcd(d,n)>1$, by Lemma \ref{multi}, the statement is true.\\
	(2) when $gcd(d,n)=1$ and $d\geq 3$, by Theorem \ref{main}, the statement is true.\\
	The remaining finite cases for smaller determinants have been checked in the beginning of the previous section and in Lemmas \ref{det1}, \ref{det21} and \ref{det22}.
\end{proof}

This completes the algorithm. The next theorem generalizes Lemma \ref{Lemma 4.4}.

\begin{theorem}\label{theorem 15.2} 
	Consider the product $(d,n)*(0,1)$ with $0\leq n\leq d > 0$. Then, 
	$$(d,n)*(0,1) = A^{2d}(d,n+1) + A^{-2d}(d,n-1) + \sum_i u_i(q_{i},p_{i}),$$ 
	with $0\leq q_i  \leq d-1$ and $0\leq p_i \leq n$.
	
\end{theorem}	

\begin{proof}
	Consider a curve (or multicurve) of slope type $\frac{n}{d}$ with $n,d >0$. In Dehn coordinates\footnote{In this proof we use notation from the book \cite{P-H}; compare Figure \ref{fig:model-curves}.}  $\frac{n}{d}$ has coordinates $(2d, n)$, that is, there are $4d$ points on Dehn annulus (that is, annulus along pants curve), say $x_1,x_2,...,x_{2d}$ on one boundary component
	and $y_1,y_2,...,y_{2d}$ on the other. Consider now the diagram $\frac{n}{d} \sqcup \frac{1}{0}$ (that is the product before skein relations). 
	We can take as $\frac{1}{0}$ curve  the boundary of  pants annulus on which there are points $x_i$. Every generic horizontal (slope $\frac{0}{1})$) curve, say $h$, intersects $\frac{n}{d} \sqcup \frac{1}{0}$ in 
	$2n+2$ points. We can always deform $h$ to a curve, say $h'$, cutting the crossing, say $x_i$, and otherwise generic, and cutting $\frac{n}{d} \sqcup \frac{1}{0}$ 
	in $2n$ additional points (there are always two choices, if $x_i$ was the first then $x_{2d+1-i}$ is another). 
	If we smooth this crossing $x_i$ in $B-type$ than the diagram obtained by this smoothing is cut by $h'$ in $2n$ points. Furthermore, for any 
	crossing $x_i$ we can find its $h$ type curve. Therefore, the first part of the lemma follows for any Kauffman state of 
	$\frac{n}{d}*\frac{1}{0}$ with at least one $B$ smoothing. We find that its deformed curve $h'$ cuts any Kauffman state of $\frac{n}{d}*\frac{1}{0}$ at most $2n$ times. Thus, any curves $(q_i,p_i) = \frac{p_i}{q_i}$ satisfy $p_i \leq n$ (in fact $|p_i|\leq n$). See Figures \ref{fig:viacrossing} and \ref{fig:viacrossingBA}.
	
\begin{figure}
	\centering
	\includegraphics[width=0.3\linewidth]{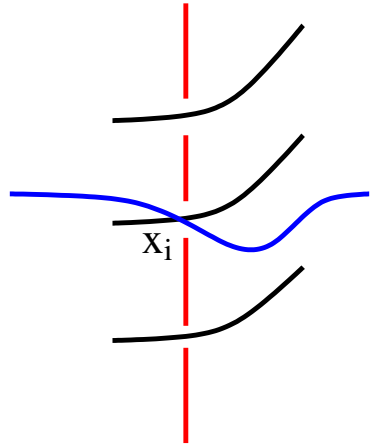}
	\caption{Passing via B-crossing.}
	\label{fig:viacrossing}
\end{figure}
	
	Similar proof works for denominator: if a Kauffman state has at least one $A$ smoothing and one $B$ smoothing then always we can find them to be neighbors 
	(along a circle) and draw $\frac{1}{0}$ curve which cuts a diagram of $\frac{n}{d} \sqcup \frac{1}{0}$ smoothed at those $A$ and $B$ crossings in no more than $2(d-1)$ points.
	Compare Figures \ref{fig:viacrossingBA} and \ref{fig:AB-smoothings-eps-converted-to.pdf}.\\
	This can be written more generally by first observing that $A$ and $B$ crossings do not have to be neighbors. Thus, we can pair $A$ and $B$ smoothings as far as some are left. 
	Thus, consider the Kauffman state with $n_A$ $A$-smoothings and $n_B=2d-n_A$ $B$-smoothings. Let $n_c$ be the minimum of these numbers. Then the corresponding $q_i$ satisfy $q_i \leq d-n_c$.	
\begin{figure}
	\centering
	\includegraphics[width=0.15\linewidth]{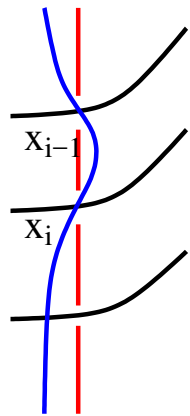}
	\caption{Passing via neighboring A and B crossings.}
	\label{fig:viacrossingBA}
\end{figure}
\end{proof} 
\begin{figure}
	\centering
	\includegraphics[width=0.4\linewidth, scale=0.5, width = 12cm]{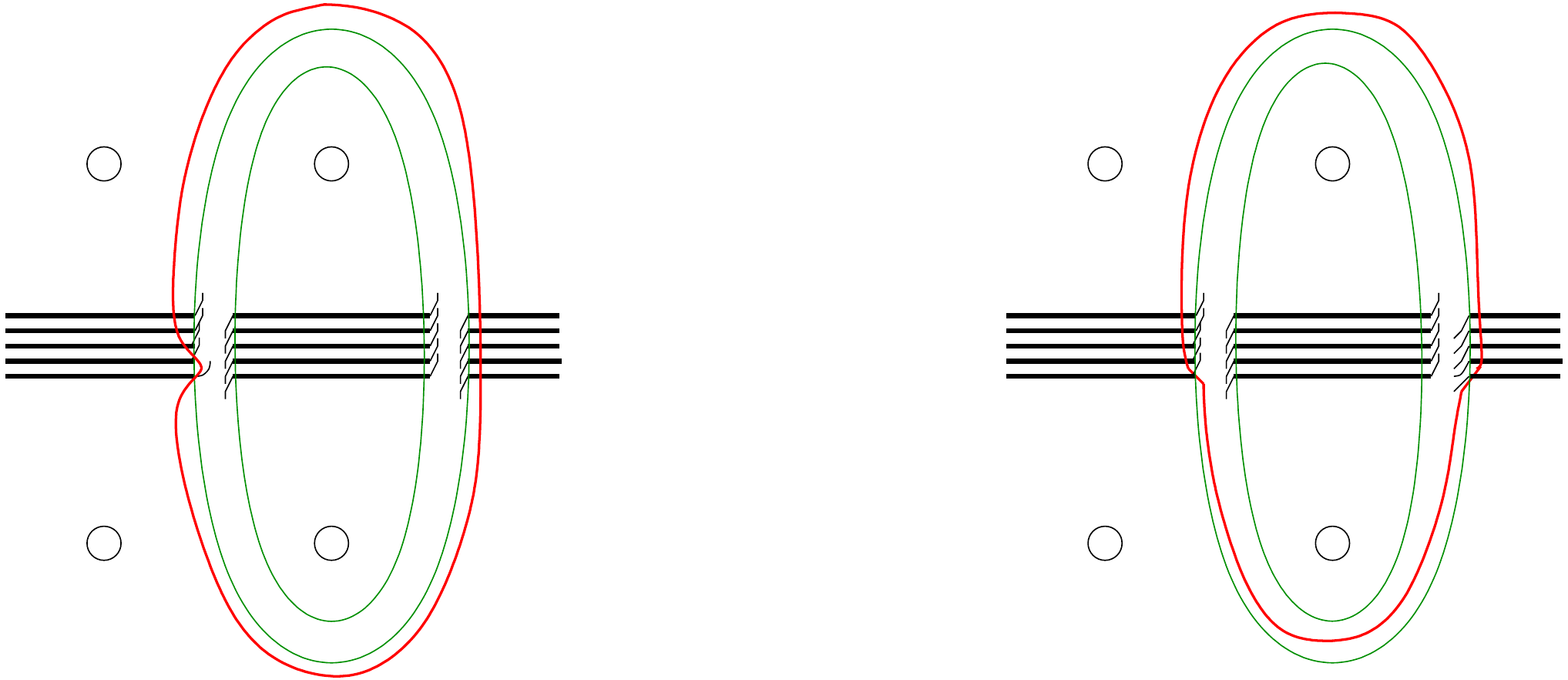}
	\caption{Examples of A-B passing.}
	\label{fig:AB-smoothings-eps-converted-to.pdf}
\end{figure}	
\begin{corollary}

Consider the product of (not necessarily reduced) positive curves $\frac{n_1}{d_1}*\frac{n_2}{d_2}$ with $d_1n_2-d_2n_1>0,$ 
	then in the skein module it can be written as 
	$$(d_1,n_1)*(d_2,n_2)= A^{2(d_1n_2-d_2n_1)}(d_1+d_2,n_1+n_2) + A^{-2(d_1n_2-d_2n_1)}(d_1-d_2,n_1-n_2) + 
	\sum_{i} u_i(q_i,p_i), $$ 
	where $0\leq p_i<n_1+n_2$ and $0\leq q_i < d_1+d_2$. 
\end{corollary}
	
	\begin{proof}
		Assume that the positive curves $(d_1,n_1), (d_2,n_2)$ are obtained from the pair $(d,n)*(0,1)$ by the action of  a nonnegative matrix of $SL(2,Z)$ (upper half of the Farey diagram). In other words, there exist $n_1$  and $d_1$ so that

\[\left[
 \begin{array}{cc}
	 n_2 &  a_{1,2}\\
	 d_2 &  a_{2,2}
 \end{array} \right]
      \left[
  \begin{array}{cc}
      n & 1 \\
      d & 0 
\end{array} \right]
     =\left[
 \begin{array}{cc}
   n_1 & n_2 \\
   d_1 & d_2 
\end{array} \right].
   \]
The condition (coming from Theorem \ref{theorem 15.2}\footnote{Let $(d,n)*(0,1)= A^{2d}(d,n+1) + A^{-2d}(d,n-1) + \sum_i u_i(q'_i,p'_i). $ According to Theorem \ref{theorem 15.2}, $0\leq p_i' \leq n$ and $0\leq q_i' \leq d-1$.  Therefore, $p_i= n_2p_i'+a_{1,2}q_i' \leq n_2n + a_{1,2}(d-1)=n_1-a_{1,2}$ and $q_i= d_2p_i' + a_{2,2}q_i'\leq d_2n+ a_{2,2}(d-1)=d_1-a_{2,2}$.}) is now: $0\leq p_i \leq n_1-a_{1,2}$, $0\leq q_i \leq d_1-a_{2,2}$.
		\end{proof}
		
The following result can be thought as a stronger version of Lemma \ref{Lemma 4.5} (2).
		
\begin{corollary}

Consider the curves $(d_{1},n_{1})$ and  $(d_{2},n_{2})$ ($d_{2}\geq 0$) and let $(d_{2},n_{2})*(0,1)=A^{2d_{2}}(d_{2},n_{2}+1)+A^{-2d_{2}}(d_{2},n_{2}-1)\sum_i w_i(q_{i},p_{i})$.  If $|n_2d_1-n_1d_2|=1$, then
		$$|n_1q-d_1p| \leq d_1.$$

\end{corollary}	

\begin{proof}

By Theorem \ref{theorem 15.2}, we have $0\leq q_{i}\leq d_{2}-1$.  From Lemma \ref{Lemma 4.4} (2) and Lemma \ref{Lemma 4.5} (1), we have 
$$|n_1q_{i}-d_1p_{i}| \leq \frac{q}{d_2}|n_2d_1-n_1d_2| + \frac{d_1}{d_2}|n_2q-d_2p|<1+d_{1}\leq d_{1}.$$ 

\end{proof}

\subsection{Motivation for Future Work }

In the future we plan to compare our result with the study of positive basis for surface skein algebras (compare with \cite{Thu}).    The third author heard about the positivity problem at the Aspen Physics Conference in March 2018 from Paul Wedrich. He heard about it from Dylan Thurston, who in turn had been motivated by a question of Edward Witten.  We ask a similar question for the KBSA of $F_{0,4} \times I$.

\begin{conjecture}\label{witten}\
\begin{enumerate}
\item[(1)] If $(d_2,n_2)_T * (d_1,n_1)_T = \sum_i w_i(q_i,p_i)_T,$
where $w_i\in K_0,$ then $w_i$ is a nonnegative polynomial in variables $A,R_{0,1},R_{1,1},R_{1,0},y$.
\item[(2)] If $(d_2,n_2)_T * (d_1,n_1)_T=A^{n_1d_2-n_2d_1}(d_1+d_2,n_1+n_2)_T + A^{-n_1d_2+n_2d_1}(d_1-d_2,n_1-n_2)_T +
\sum_i v_i(q_i,p_i)_S$, with $v_i\in K_0$, then $v_i$ are nonnegative polynomials in variables $A,R_{0,1},R_{1,1},R_{1,0},y$.
\end{enumerate}
\end{conjecture}

Notice that by Proposition \ref{property_cheby} (4), Conjecture \ref{witten}, (2) is stronger than (1).  All our computations and the formulas in Theorem \ref{Lemma 14.1} and \ref{thm:xiao} support the former conjecture.

\begin{question}
\begin{enumerate}\
	\item {What is  the complexity of multiplying fractions in the KBSA of $F_{0,4} \times I$?}
	\item{Is the complexity of the former algorithm quasi-polynomial in terms of the determinant of the curves being multiplied (that is, of complexity $det^{log(det)}$ where det is the determinant of a pair of fractions)? }
\end{enumerate}	
\end{question}	
	
\section*{Appendix: Proof of Theorem \ref{thm:xiao}}
  
Before we prove Theorem \ref{thm:xiao}, we provide one instructive example $(7,1)*(0,1)$.

$$(7,1)*(0,1)= A^{14}(7,2)_T + A^{-14}(7,0)_T + $$ 
$$y(A^{-10}(5,0)_S + A^{-2}(3,0)_S + A^6(1,0)) +$$
$$A^2[R_{1,1}(6,1) + [2]_{A^4}R_{0,1}(5,1) + [3]_{A^4}R_{1,1}(4,1) + [3]_{A^4}R_{0,1}(3,1) + [2]_{A^4}R_{1,1}(2,1)+R_{0,1}(1,1)]+$$
$$R_{1,0}[A^{12} + A^4(2,0)_S + A^{-4}(4,0)_S + A^{-12}(6,0)_S] +$$
$$(R_{1,0}+ R_{1,1}R_{0,1})[A^8+ (2,0)_S + A^{-8}(4,0)_S] +$$
$$(R_{1,1}^2+R_{0,1}^2)[A^2(1,0) + A^{-6}(3,0)_S + 2A^{-2}(1,0) ] + $$
$$+5R_{0,1}R_{1,1} + 3R_{0,1}R_{1,1}[A^{-4}(2,0)_S +A^4].$$

\begin{thm:xiao}
	
Let $n$ be a non-negative integer. Then, 
$$(n,1)*(0,1)=A^{2n}(n,2)_T + A^{-2n}(n,0)_T + \sum_{i=0}^{n-1}\alpha_{n-1}\beta_i +A^2\sum_{i=1}^{n-1}[n_i]_{A^4} \frac{1}{i}R_{i-1,1},$$
where
\[
n_i=
\left\{
\begin{array}{ll}
\left[i\right]_{A^4} &  \mbox{for $ i\leq \frac{n}{2}, $}\\
\left[n-i\right]_{A^4} & \mbox{for $i\geq \frac{n}{2},$}
\end{array}
\right.
\]

\begin{minipage}[t]{0.45\textwidth}
	$$ \alpha_{i}=\left\{
		\begin{array}{rcl}
		R_{1,0}       &      & {i=1,}\\
		y     &      & {i=2,}\\
		R_{0,1}R_{1,1}+R_{1,0}     &      & {i=3,}\\
		\frac{(i-2)}{2}(R_{0,1}^{2}+R_{1,1}^{2})       &      & {i>3  \mbox{ even,}}\\
		(i-2)R_{0,1}R_{1,1}       &      & {i>3 \mbox{ odd, }} 
		\end{array} \right. $$
\end{minipage}
\begin{minipage}[t]{0.45\textwidth}
	$$\mbox {and } \ \beta_{i}=\left\{
		\begin{array}{lll}
		\sum\limits_{j=0}^{(i-1)/2} A^{-2i+8j}(i-2j,0)_{S}      &      & {i  \mbox{ odd,}}\\
		\sum\limits_{j=0}^{i/2} A^{-2i+8j}(i-2j,0)_{S}        &      & {i \mbox{ even}}.
		\end{array} \right. $$
\end{minipage}
\ \\

Notice that all the coefficients are positive and unimodal.  

\end{thm:xiao}

\begin{spacing}{2}
	
\begin{proof}
	
	Clearly the formula is true for $n=0,1,2$ using initial data.  Assume the formula is true for $n<N$. We prove the case $n=N$ by induction. First, consider the case when $N$ is even. Then, $(n,1)*(0,1)=A^{-2}(1,0)*(n-1,1)*(0,1)-A^{-2}R_{n,1}(0,1)-A^{-4}(n-2,1)*(0,1).$  Now, we consider each term separately.
	
	The simpler one is the second product:
	
	$(n-2,1)*(0,1)=A^{2(n-2)}(n-2,2)_{T}+A^{-2(n-2)}(n-2,0)_{T}+\sum\limits_{i=1}^{(n-4)/2}A^{2}[i]_{A^4}R_{i,1}(i,1)+\\ A^{2}[(n-2)/2]_{A^4}R_{(n-2)/2,1}((n-2)/2,1)+\sum\limits_{i=n/2}^{n-3}A^{2}[n-2-i]_{A^4}R_{i,1}(i,1)+\sum\limits_{i=0}^{n-3}\alpha_{n-2-i}\beta_{i}.$
	
	Now we continue with the computation of the first product of the triple.  By the induction hypothesis, we know the product of $(n-1,1)*(0,1)$. Thus,  $(1,0)*(n-1,1)*(0,1)=A^{2(n-1)}(1,0)*(n-1,2)_{T}+\\ A^{-2(n-1)}(1,0)*(n-1,0)_{T}+\sum\limits_{i=1}^{(n-2)/2}A^{2}[i]_{A^4}R_{i-1,1}(1,0)*(i,1)+\sum\limits_{i=(n-2)/2+1}^{n-2}A^{2}[n-1-i]_{A^4}R_{i-1,1}(1,0)*(i,1)+\\
	\sum\limits_{i=0}^{n-2}\alpha_{n-1-i}(1,0)*\beta_{i}=A^{2(n-1)}[A^{4}(n,2)_{T}+A^2R_{n/2,1}(n/2,1)+y+A^{-2}R_{(n-2)/2,1}((n-2)/2,1)+\\ A^{-4}(n-2,2)_{T}]+A^{-2(n-1)}(1,0)*(n-1,0)_{T}+\sum\limits_{i=1}^{(n-2)/2}A^{2}[i]_{A^4}R_{i-1,1}(1,0)*(i,1)+\\
	\sum\limits_{i=(n-2)/2+1}^{n-2}A^{2}[n-1-i]_{A^4}R_{i-1,1}(1,0)*(i,1)+\sum\limits_{i=0}^{n-2}\alpha_{n-2-i}(1,0)*\beta_{i}.$
	
	Combining the two parts, we have
	
	$(n,1)*(0,1)=A^{-2}(1,0)*(n-1,1)*(0,1)-A^{-2}R_{n,1}(0,1)-A^{-4}(n-2,1)*(0,1)=A^{2n}(n,2)_{T}+A^{2n-2}R_{n/2,1}(n/2,1)+A^{2n-4}y+A^{2n-6}R_{(n-2)/2,1}((n-2)/2,1)+A^{2n-8}(n-2,2)_{T}+{\color{red}A^{-2n}(1,0)*(n-1,0)_{T}}+{\color{green}\sum\limits_{i=1}^{(n-2)/2}[i]_{A^4}R_{i-1,1}[A^{2}(i+1,1)+R_{i+1,1}+A^{-2}(i-1,1)]}+\\{\color{blue}\sum\limits_{i=(n-2)/2+1}^{n-2}[n-1-i]_{A^4}R_{i-1,1}[A^{2}(i+1,1)+R_{i+1,1}+A^{-2}(i-1,1)]}+\\
	\sum\limits_{i=0}^{n-2}A^{-2}\alpha_{n-1-i}(1,0)*\beta_{i}-A^{2(n-4)}(n-2,2)_{T}-{\color{red}A^{-2n}(n-2,0)_{T}}-\sum\limits_{i=1}^{(n-4)/2}A^{-2}[i]_{A^4}R_{i,1}(i,1)-\\A^{-2}[(n-2)/2]_{A^4}R_{(n-2)/2,1}((n-2)/2,1)-\sum\limits_{i=n/2}^{n-3}A^{-2}[n-2-i]_{A^4}R_{i,1}(i,1)-\sum\limits_{i=0}^{n-3}A^{-4}\alpha_{n-2-i}\beta_{i}-A^{-2}R_{n,1}(0,1)\\
	=A^{2n}(n,2)_{T}+{\color{red}A^{-2n}((1,0)*(n-1,0)_{T}-(n-2,0)_{T})}+A^{2n-2}R_{n/2,1}(n/2,1)+A^{2n-4}y+\\ A^{2n-6}R_{(n-2)/2,1}((n-2)/2,1)+A^{2n-8}(n-2,2)_{T}+
	{\color{green} [ \sum\limits_{i=1}^{(n-4)/2}A^{2}[i]_{A^4}R_{i+1,1}(i+1,1)+}\\ {\color{green} 
	A^{2}[(n-2)/2]_{A^{4}}R_{n/2,1}(n/2,1)]+\sum\limits_{i=1}^{(n-2)/2}[i]_{A^4}R_{i-1,1}R_{i+1,1}+[\sum\limits_{k=1}^{(n-4)/2}[k+1]_{A^4}R_{k,1}A^{-2}(k,1)+}\\ {\color{green}
	A^{-2}R_{0,1}(0,1)]}+{\color{blue}\sum\limits_{k=(n+2)/2}^{n-1}[n-k]_{A^4}R_{k,1}A^{2}(k,1)+	\sum\limits_{i=n/2}^{n-2}[n-1-i]_{A^4}R_{i-1,1}R_{i+1,1}+}\\
	{\color{blue}\sum\limits_{k=(n-2)/2}^{n-3}A^{-2}[n-2-k]_{A^4}R_{k,1}(k,1)}+\sum\limits_{i=0}^{n-2}A^{-2}\alpha_{n-1-i}(1,0)*\beta_{i}-A^{2n-8}(n-2,2)_{T}-\sum\limits_{i=1}^{(n-4)/2}A^{-2}[i]_{A^4}R_{i,1}(i,1)-A^{-2}[(n-2)/2]_{A^4}R_{(n-2)/2,1}((n-2)/2,1)-\sum\limits_{i=n/2}^{n-3}A^{-2}[n-2-i]_{A^4}R_{i,1}(i,1)-\sum\limits_{i=0}^{n-3}A^{-4}\alpha_{n-2-i}\beta_{i}-A^{-2}R_{n,1}(0,1).$
	
	By the recursive relation of Chebyshev polynomial, we have $(n,1)*(0,1)=\\
	A^{2n}(n,2)_{T}+A^{-2n}(n,0)_{T}+{\color{green}A^{2n-2}R_{n/2,1}(n/2,1)}+A^{2n-4}y+A^{2n-6}R_{(n-2)/2,1}((n-2)/2,1)+{\color{blue}A^{2n-8}(n-2,2)_{T}}+[{\color{red}\sum\limits_{i=1}^{(n-4)/2}A^{2}[i]_{A^4}R_{i+1,1}(i+1,1)}+{\color{green}A^{2}[(n-2)/2]_{A^{4}}R_{n/2,1}(n/2,1)}]+\sum\limits_{i=1}^{(n-2)/2}[i]_{A^4}R_{i-1,1}R_{i+1,1}+\\ {\color{darkgreen}[\sum\limits_{k=1}^{(n-4)/2}[k+1]_{A^4}R_{k,1}A^{-2}(k,1)}
	+{\color{blue}A^{-2}R_{0,1}(0,1)}]+{\color{cyan}\sum\limits_{k=(n+2)/2}^{n-1}[n-k]_{A^4}R_{k,1}A^{2}(k,1)}+\\ \sum\limits_{i=n/2}^{n-2}[n-1-i]_{A^4}R_{i-1,1}R_{i+1,1}+
	{\color{blue}\sum\limits_{k=(n-2)/2}^{n-3}A^{-2}[n-2-k]_{A^4}R_{k,1}(k,1)}+ \sum\limits_{i=0}^{n-2}A^{-2}\alpha_{n-1-i}(1,0)*\beta_{i}-\\ {\color{blue}A^{2n-8}(n-2,2)_{T}}-{\color{darkgreen}\sum\limits_{i=1}^{(n-4)/2}A^{-2}[i]_{A^4}R_{i,1}(i,1)}{\color{blue}-A^{-2}[(n-2)/2]_{A^4}R_{(n-2)/2,1}((n-2)/2,1)-} \\ {\color{blue} \sum\limits_{i=n/2}^{n-3}A^{-2}[n-2-i]_{A^4}R_{i,1}(i,1)}-\sum\limits_{i=0}^{n-3}A^{-4}\alpha_{n-2-i}\beta_{i}-{\color{blue}A^{-2}R_{n,1}(0,1)}.$
	
	The blue terms cancel in pairs and we have $(n,1)*(0,1)=$ \\
	$A^{2n}(n,2)_{T}+A^{-2n}(n,0)_{T}+{\color{green}A^{2}[A^{2n-4}+[(n-2)/2]_{A^{4}}]R_{n/2,1}(n/2,1)}+A^{2n-4}y +$ \\ 
	$A^{2n-6}R_{(n-2)/2,1}((n-2)/2,1)+ {\color{red}[\sum\limits_{k=2}^{(n-4)/2}A^{2}[k-1]_{A^4}R_{k,1}(k,1)+A^{2}[n/2-2]_{A^{4}}R_{(n-2)/2,1}((n-2)/2,1)]}+$ \\ $\sum\limits_{i=1}^{(n-2)/2}[i]_{A^4}R_{i-1,1}R_{i+1,1}+\sum\limits_{i=n/2}^{n-2}[n-1-i]_{A^4}R_{i-1,1}R_{i+1,1}+$
	${\color{darkgreen}[\sum\limits_{k=1}^{(n-4)/2}A^{-2}[k+1]_{A^4}R_{k,1}(k,1)-} \\
	{\color{darkgreen}\sum\limits_{i=1}^{(n-4)/2}A^{-2}[i]_{A^4}R_{i,1}(i,1)]}+{\color{cyan}\sum\limits_{k=(n+2)/2}^{n-1}A^{2}[n-k]_{A^4}R_{k,1}(k,1)}+{\color{blue}[\sum\limits_{k=(n-2)/2}^{n-3}A^{-2}[n-2-k]_{A^4}R_{k,1}(k,1)-} \\ {\color{blue}\sum\limits_{i=n/2}^{n-3}A^{-2}[n-2-i]_{A^4}R_{i,1}(i,1)-A^{-2}[n-2/2]_{A^4}R_{(n-2)/2,1}((n-2)/2,1)]}+\\ \sum\limits_{i=0}^{n-2}A^{-2}\alpha_{n-1-i}(1,0)*\beta_{i}-\sum\limits_{i=0}^{n-3}A^{-4}\alpha_{n-2-i}\beta_{i}=
	A^{2n}(n,2)_{T}+A^{-2n}(n,0)_{T}+{\color{green}A^{2}[n/2]_{A^{4}}R_{n/2,1}(n/2,1)}+\\ A^{2n-4}y+A^{2n-6}R_{(n-2)/2,1}((n-2)/2,1)+[\sum\limits_{k=2}^{(n-4)/2}A^{2}[k-1]_{A^4}R_{k,1}(k,1)+A^{2}[n/2-2]_{A^{4}}R_{(n-2)/2,1}((n-2)/2,1)]
	+\\\sum\limits_{i=1}^{(n-2)/2}[i]_{A^4}R_{i-1,1}R_{i+1,1}+\sum\limits_{i=n/2}^{n-2}[n-1-i]_{A^4}R_{i-1,1}R_{i+1,1}+{\color{darkgreen}\sum\limits_{k=1}^{(n-4)/2}A^{2}A^{4k-4}R_{k,1}(k,1)}+\\ \sum\limits_{k=(n+2)/2}^{n-1}A^{2}[n-k]_{A^4}R_{k,1}(k,1)+\sum\limits_{i=0}^{n-2}A^{-2}\alpha_{n-1-i}(1,0)*\beta_{i}-\sum\limits_{i=0}^{n-3}A^{-4}\alpha_{n-2-i}\beta_{i}=$
	
	$A^{2n}(n,2)_{T}+A^{-2n}(n,0)_{T}+A^{2}[n/2]_{A^{4}}R_{n/2,1}(n/2,1)+A^{2n-4}y+[{\color{darkgreen}A^{2}R_{1,1}(1,1)}{\color{darkgreen}+\sum\limits_{k=2}^{(n-4)/2}A^{2}A^{4k-4}R_{k,1}(k,1)}+\sum\limits_{k=2}^{(n-4)/2}A^{2}[k-1]_{A^4}R_{k,1}(k,1)+A^{2}[n/2-2]_{A^{4}}R_{(n-2)/2,1}((n-2)/2,1)+A^{2n-6}R_{(n-2)/2,1}((n-2)/2,1)]+\sum\limits_{i=1}^{(n-2)/2}[i]_{A^4}R_{i-1,1}R_{i+1,1}+\sum\limits_{i=n/2}^{n-2}[n-1-i]_{A^4}R_{i-1,1}R_{i+1,1}+\sum\limits_{k=(n+2)/2}^{n-1}A^{2}[n-k]_{A^4}R_{k,1}(k,1)
	+\\\sum\limits_{i=0}^{n-2}A^{-2}\alpha_{n-1-i}(1,0)*\beta_{i}-\sum\limits_{i=0}^{n-3}A^{-4}\alpha_{n-2-i}\beta_{i}.$
	
	Thus, we have $(n,1)*(0,1)=\\
	A^{2n}(n,2)_{T}+A^{-2n}(n,0)_{T}+A^{2}[n/2]_{A^{4}}R_{n/2,1}(n/2,1)+A^{2n-4}y+{\color{darkgreen}[A^{2}R_{1,1}(1,1)}
	{\color{darkgreen}+\sum\limits_{k=2}^{(n-4)/2}A^{2}A^{4k-4}R_{k,1}(k,1)+}\\{\color{darkgreen}\sum\limits_{k=2}^{(n-4)/2}A^{2}[k-1]_{A^4}R_{k,1}(k,1)+A^{2}[n/2-2]_{A^{4}}R_{(n-2)/2,1}((n-2)/2,1)+A^{2n-6}R_{(n-2)/2,1}((n-2)/2,1)]}+ \\\sum\limits_{i=1}^{(n-2)/2}[i]_{A^4} 
	R_{i-1,1}R_{i+1,1}+\sum\limits_{i=n/2}^{n-2}[n-1-i]_{A^4}R_{i-1,1}R_{i+1,1}+{\color{red}\sum\limits_{k=(n+2)/2}^{n-1}A^{2}[n-k]_{A^4}R_{k,1}(k,1)}
	+\\ \sum\limits_{i=0}^{n-2}A^{-2}\alpha_{n-1-i}(1,0)*\beta_{i}-\sum\limits_{i=0}^{n-3}A^{-4}\alpha_{n-2-i}\beta_{i} =$ 
	$\\
	A^{2n}(n,2)_{T}+A^{-2n}(n,0)_{T}+{\color{darkgreen}\sum\limits_{i=1}^{(n-2)/2}A^{2}[i]_{A^{4}}R_{i,1}(i,1)}+A^{2}[n/2]_{A^{4}}R_{n/2,1}(n/2,1)+
	{\color{red}\sum\limits_{i=(n+2)/2}^{(n-1)}
		A^{2}[n-i]_{A^4}R_{i,1}(i,1)} + \\A^{2n-4}y+\sum\limits_{i=1}^{(n-2)/2}[i]_{A^4}R_{i-1,1}R_{i+1,1}+\sum\limits_{i=n/2}^{n-2}[n-1-i]_{A^4}R_{i-1,1}R_{i+1,1}+\sum\limits_{i=0}^{n-2}A^{-2}\alpha_{n-1-i}(1,0)*\beta_{i}-\sum\limits_{i=0}^{n-3}A^{-4}\alpha_{n-2-i}\beta_{i}.$
	
	Now, let us consider $\sum\limits_{i=0}^{n-2}A^{-2}\alpha_{n-1-i}(1,0)*\beta_{i}-\sum\limits_{i=0}^{n-3}A^{-4}\alpha_{n-2-i}\beta_{i}=\\A^{-2}\alpha_{n-1}(1,0)+\sum\limits_{i=1}^{n-2}\alpha_{n-1-i}[A^{-2}(1,0)*\beta_{i}-A^{-4}\beta_{i-1}].$ Now, \small{$$ A^{-2}(1,0)*\beta_{i}-A^{-4}\beta_{i-1}=\left\{
		\begin{array}{rcl}
		\beta_{i+1}      &      & {i\ even,}\\
		\beta_{i+1}-A^{2i+2}       &      & {i\ odd.}
		\end{array} \right. $$} 
	
So, $\sum\limits_{i=0}^{n-2}A^{-2}\alpha_{n-1-i}(1,0)*\beta_{i}-\sum\limits_{i=0}^{n-3}A^{-4}\alpha_{n-2-i}\beta_{i}=A^{-2}\alpha_{n-1}(1,0)+\sum\limits_{i=1}^{n-2}\alpha_{n-1-i}\beta_{i+1}-\sum\limits_{k=1}^{(n-2)/2}\alpha_{n-1-(2i-1)}A^{2(2i-1)+2}=A^{-2}\alpha_{n-1}(1,0)+\sum\limits_{i=2}^{n-1}\alpha_{n-i}\beta_{i}-\sum\limits_{i=1}^{(n-2)/2}\alpha_{n-1-(2i-1)}A^{2(2i-1)+2}=A^{-2}\alpha_{n-1}(1,0)+\sum\limits_{i=2}^{n-1}\alpha_{n-i}\beta_{i}-\sum\limits_{i=1}^{(n-2)/2}\alpha_{n-2i}A^{4i}=\sum\limits_{i=1}^{n-1}\alpha_{n-i}\beta_{i}-\sum\limits_{i=1}^{(n-2)/2}\alpha_{n-2i}A^{4i}.$
	
	Thus,
	
	$A^{2n}(n,2)_{T}+A^{-2n}(n,0)_{T}+\sum\limits_{i=1}^{(n-2)/2}A^{2}[i]_{A^{4}}R_{i,1}(i,1)+A^{2}[n/2]_{A^{4}}R_{n/2,1}(n/2,1)+\sum\limits_{i=(n+2)/2}^{n-1}A^{2}[n-i]_{A^4}R_{i,1}(i,1)+A^{2n-4}y+\sum\limits_{i=1}^{(n-2)/2}[i]_{A^4}R_{i-1,1}R_{i+1,1}+\sum\limits_{i=n/2}^{n-2}[n-1-i]_{A^4}R_{i-1,1}R_{i+1,1}+\sum\limits_{i=0}^{n-2}A^{-2}\alpha_{n-1-i}(1,0)*\beta_{i}-\sum\limits_{i=0}^{n-3}A^{-4}\alpha_{n-2-i}\beta_{i}=$
	
	$A^{2n}(n,2)_{T}+A^{-2n}(n,0)_{T}+\sum\limits_{i=1}^{(n-2)/2}A^{2}[i]_{A^{4}}R_{i,1}(i,1)+A^{2}[n/2]_{A^{4}}R_{n/2,1}(n/2,1)+\sum\limits_{i=(n+2)/2}^{n-1}A^{2}[n-i]_{A^4}R_{i,1}(i,1)+\sum\limits_{i=1}^{n-1}\alpha_{n-i}\beta_{i}+A^{2n-4}y+\sum\limits_{i=1}^{(n-2)/2}[i]_{A^4}R_{i-1,1}R_{i+1,1}+\sum\limits_{i=n/2}^{n-2}[n-1-i]_{A^4}R_{i-1,1}R_{i+1,1}-\sum\limits_{i=1}^{(n-2)/2}\alpha_{n-2i}A^{4i}=$
	
	$A^{2n}(n,2)_{T}+A^{-2n}(n,0)_{T}+\sum\limits_{i=1}^{(n-2)/2}A^{2}[i]_{A^{4}}R_{i,1}(i,1)+A^{2}[n/2]_{A^{4}}R_{n/2,1}(n/2,1)+\sum\limits_{i=(n+2)/2}^{n-1}A^{2}[n-i]_{A^4}R_{i,1}(i,1)+\sum\limits_{i=0}^{n-1}\alpha_{n-i}\beta_{i}.$
	
	The case when $n=N$ is odd can be proved in a similar way. 
\end{proof}

\end{spacing}

\section*{Acknowledgements}
J\'{o}zef H. Przytycki was partially supported by the Simons Foundation Collaboration Grant for Mathematicians - 316446 and CCAS Dean's Research Chair award. Marithania Silvero was partially supported by MTM2016-76453-C2-1-P and FEDER, and acknowledges financial support from the Spanish Ministry of Economy and Competitiveness, through the Mar\'ia de Maeztu Programme for Units of Excellence in R\&D (MDM-2014-0445).

\end{document}